\documentclass{amsart}
\usepackage{amsfonts,hyperref}

\newtheorem{theorem}{Theorem}
\theoremstyle{plain}

\newtheorem{corollary}{Corollary}

\newtheorem{definition}{Definition}

\newtheorem{lemma}{Lemma}

\numberwithin{equation}{section}
\theoremstyle{theorem}

\begin{document}
\title[A note on subordination]{A note on subordination}
\author{Andr\'e Gomes}
\address{%
\parbox{.9\textwidth}{\centering
Department of Mathematics\\
Institute of Mathematics and Statistics\\
University of S\~ao Paulo\\
S\~ao Paulo - 05508-090 - SP\\
Brazil\\[5pt]}}
\email{gomes@ime.usp.br}

\begin{abstract}
In this paper we present a result about analytic functions $f$ defined on the open unit disc and with a finite number of exceptional values contained in the real interval $(0,1)$. We find an upper bound for $\left|f'(0)\right|$. This bound is sharp in the case of one exceptional value. We also analyzed the case of two exceptional values.
\end{abstract}

\maketitle

\section{Preliminaries}
The definition and results of this section can be seen in [1] (p. 450).
\begin{definition} 
Let $\mathbb{D}=\left\{z\in\mathbb{C};\ \left|z\right|<1\right\}$ be the open unit disk in the complex plane. Let $f$ and $g$ be analytic functions on $\mathbb{D}$. The function $f$ is said to be \textit{subordinate} to $g$ (and we denote this relation by $f\prec g$) if there exists an analytic function $\varphi:\mathbb {D}\longrightarrow \mathbb{D}$ satisfying $\varphi(0)=0$ and such that $f=g\circ\varphi$. 
\end{definition} 
\vspace{0.2cm}
The following two lemmas are immediate consequences of the Definition 1.
\\
\begin{lemma}\label{le1} 
If $f\prec g$, then $\left|f'(0)\right|\leq \left|g'(0)\right|$.
\end{lemma}
\begin{proof}
We have $\left|f'(0)\right|=\left|g'(0)\right|\left|\varphi'(0)\right|$ and (by Schwarz's Lemma) $\left|\varphi'(0)\right|\leq 1$.
\end{proof}
\vspace{0.1cm}
\begin{lemma}\label{le2}
Denote by $\mathcal{P}$ the collection of all analytic functions $f:\mathbb{D}\longrightarrow \mathbb{C}$ such that $f(0)=1$ and $\mathrm{Re}(f)>0$. If $f\in \mathcal{P}$ and $F:\mathbb{D}\longrightarrow\left\{z\in\mathbb{C};\ \mathrm{Re}(z)>0\right\}$ is the function given by
\begin{eqnarray*}
	F(z)=\frac{1+z}{1-z}
\end{eqnarray*}
then $f\prec F$.
\end{lemma}
\begin{proof}
It is sufficient to define $\varphi:\mathbb{D}\longrightarrow \mathbb{D}$ by $\varphi(z)=F^{-1}(f(z))= \frac{f(z)-1}{f(z)+1}$.
\end{proof}

\section{Main results}

\begin{theorem}
Let $f:\mathbb{D}\longrightarrow \mathbb{D}$ be an analytic function with $f(0)=0$. If there exists real numbers $\left\{\alpha_{1},\alpha_{2},\ldots,\alpha_{k}\right\}\subset (0,1)$ such that $f(\mathbb{D})\cap \left\{\alpha_{1},\alpha_{2},\ldots,\alpha_{k}\right\}=\emptyset$. Then
\begin{equation}\label{eq1}
\left|f'(0)\right|\leq\frac{2\ln(\frac{1}{\alpha_{1}\cdot\alpha_{2}\cdot\ldots\cdot\alpha_{k}})}{[\frac{1-\alpha_{1}^{2}}{\alpha_{1}}+\ldots+\frac{1-\alpha_{k}^{2}}{\alpha_{k}}]}
\end{equation}
\end{theorem}

\begin{proof}
We define the analytic function $h:\mathbb{D}\longrightarrow \mathbb{C}-\left\{0\right\}$ given by the following expression
\begin{eqnarray*}	h(z)=\left(\frac{\alpha_{1}-f(z)}{1-\alpha_{1}f(z)}\right)\cdot\left(\frac{\alpha_{2}-f(z)}{1-\alpha_{2}f(z)}\right)\cdot\ldots\cdot\left(\frac{\alpha_{k}-f(z)}{1-\alpha_{k}f(z)}\right)
\end{eqnarray*}
The open disk $\mathbb{D}$ is simply connected and then there exists an analytic function $g:\mathbb{D}\longrightarrow \mathbb{C}$ such that $\mathrm{e}^{g(z)}=h(z)$ for all $z\in\mathbb{D}$. Taking modulus on both sides of $\mathrm{e}^{g}=h$ we obtain
\begin{eqnarray*}	\mathrm{e}^{\mathrm{Re}(g(z))}=\left|\frac{\alpha_{1}-f(z)}{1-\alpha_{1}f(z)}\right|\cdot\ldots\cdot\left|\frac{\alpha_{k}-f(z)}{1-\alpha_{k}f(z)}\right|,\ \ \ \ \ \forall z\in \mathbb{D}
\end{eqnarray*}
If we take the logarithms of both sides of the last expression, we obtain
\begin{eqnarray*}	\mathrm{Re}(g(z))=\ln\left\{\left|\frac{\alpha_{1}-f(z)}{1-\alpha_{1}f(z)}\right|\cdot\ldots\cdot\left|\frac{\alpha_{k}-f(z)}{1-\alpha_{k}f(z)}\right|\right\},\ \ \ \ \ \forall z\in \mathbb{D}
\end{eqnarray*}
We note that the expression for $\mathrm{Re}(g(z))$ is negative for all $z\in\mathbb{D}$. Define now 
\begin{eqnarray*}
\rho:\mathbb{D}\longrightarrow \mathbb{C},\ \ \ \ \ \rho(z)=\frac{g(z)}{\ln(\alpha_{1})+\ldots+\ln(\alpha_{k})}	
\end{eqnarray*}
We claim that $\rho \in \mathcal{P}$. In fact, it's easy to see that the function $\rho$ is analytic, $\rho(0)=1$ and $\mathrm{Re}(\rho)>0$. Hence, according to Lemma 2, we have $\rho\prec F$. The Lemma 1 gives us $\left|\rho'(0)\right|\leq \left|F'(0)\right|=2$. Now, we take the derivative of both sides of $\mathrm{e}^{g}=h$. We obtain
\begin{eqnarray*}	g'(z)\!\!\!&=&\!\!\!\mathrm{e}^{-g(z)}\cdot\left\{\sum\limits_{j=1}^{k}(\frac{\alpha_{1}-f(z)}{1-\alpha_{1}f(z)})\cdot\ldots\cdot(\frac{\alpha_{j}-f(z)}{1-\alpha_{j}f(z)})'\cdot\ldots\cdot(\frac{\alpha_{k}-f(z)}{1-\alpha_{k}f(z)})\right\}\\
&=&\!\!\!\mathrm{e}^{-g(z)}\cdot\left\{\sum\limits_{j=1}^{k}(\frac{\alpha_{1}-f(z)}{1-\alpha_{1}f(z)})\cdot\ldots\cdot(\frac{(\alpha_{j}^{2}-1)f'(z)}{(1-\alpha_{j}f(z))^{2}})\cdot\ldots\cdot(\frac{\alpha_{k}-f(z)}{1-\alpha_{k}f(z)})\right\}
\end{eqnarray*}
and then $g'(0)=f'(0)\cdot\left\{\frac{\alpha_{1}^{2}-1}{\alpha_{1}}+\ldots+\frac{\alpha_{k}^{2}-1}{\alpha_{k}}\right\}$ (because of $\mathrm{e}^{g(0)}=\alpha_{1}\cdot\ldots\cdot\alpha_{k}$). Then,
\begin{eqnarray*}	\left|\rho'(0)\right|=\frac{\left|f'(0)\right|}{\left|\ln(\alpha_{1})+\ldots+\ln(\alpha_{k})\right|}\cdot\left[\frac{1-\alpha_{1}^{2}}{\alpha_{1}}+\ldots+\frac{1-\alpha_{k}^{2}}{\alpha_{k}}\right]
\end{eqnarray*}
Then the inequality (2.1) is a consequence of $\left|\rho'(0)\right|\leq 2$.
\end{proof}

\begin{corollary}\upshape{[2, p. 233]} \textit{Suppose that $\alpha\in (0,1)$. If $f:\mathbb{D}\longrightarrow \mathbb{D}$ is an analytic function with $f(0)=0$ and $f(z)\neq\alpha$ for all $z\in\mathbb{D}$, then we have the following inequality $\left|f'(0)\right|\leq 2\alpha\ln(1/\alpha)/(1-\alpha^{2})$. This inequality is sharp in the sense that there exists a function $f$ such that equality holds.}
\end{corollary}
\begin{proof}
The first part of the corollary is obvious. It follows from our proof of Theorem 1 that the equality occurs when $\left|\rho'(0)\right|=2$. It is equivalent to say that $\left|\varphi'(0)\right|=1$, where $\varphi:\mathbb{D}\longrightarrow \mathbb{D}$ is the function $\varphi(z)=F^{-1}(\rho(z))=\frac{g(z)-\ln(\alpha)}{g(z)+\ln(\alpha)}$. By Schwarz's Lemma, $\left|\varphi'(0)\right|=1$ if and only if there exists a complex $c$ such that $\left|c\right|=1$ and $\varphi(z)=cz$. Then, 
\begin{eqnarray*}
	\left|\varphi'(0)\right|=1\!\!&\Longleftrightarrow& \!\!\varphi(z)=cz\Longleftrightarrow  g(z)=(\frac{1+cz}{1-cz})\cdot \ln(\alpha)\\
&\Rightarrow&\!\!\frac{\alpha-f(z)}{1-\alpha f(z)}=\mathrm{e}^{g(z)}=\mathrm{e}^{(\frac{1+cz}{1-cz})\ln(\alpha)}\\
&\Rightarrow&\!\!f(z)=\frac{\alpha-\mathrm{e}^{(\frac{1+cz}{1-cz})\ln(\alpha)}}{1-\alpha\mathrm{e}^{(\frac{1+cz}{1-cz})\ln(\alpha)}},\ \ \ \ \forall z\in\mathbb{D}
\end{eqnarray*}
The function $f$ satisfies $f(0)=0$, $\left|f\right|<1$  and $\left|f'(0)\right|=2\alpha\ln(1/\alpha)/(1-\alpha^{2})$. \end{proof}
We analyze now the case of two distinct exceptional values $\alpha_{1},\alpha_{2}\in(0,1)$. With the notations of Corollary 1, the equality in $(2.1)$ occurs if and only if 
\begin{eqnarray*}
	g(z)=[\frac{1+cz}{1-cz}]\ln(\alpha_{1}\alpha_{2})
\end{eqnarray*}
Using the identity $h=\mathrm{e}^{g}$ (where h is as in Theorem 1), we can write
\begin{eqnarray*}
\left(\frac{\alpha_{1}-f(z)}{1-\alpha_{1}f(z)}\right)\cdot\left(\frac{\alpha_{2}-f(z)}{1-\alpha_{2}f(z)}\right)=\mathrm{e}^{[\frac{1+cz}{1-cz}]\ln(\alpha_{1}\alpha_{2})}\\
\end{eqnarray*}
By expanding the above expression, we obtain 
\begin{eqnarray*}
	(1-\alpha_{1}\alpha_{2}u)f^{2}+(u-1)(\alpha_{1}+\alpha_{2})f+\alpha_{1}\alpha_{2}-u=0
\end{eqnarray*}
where $u:\mathbb{D}\longrightarrow\mathbb{D}-\left\{0\right\}$ is the function $u(z)=\mathrm{e}^{[\frac{1+cz}{1-cz}]\ln(\alpha_{1}\alpha_{2})}$.
\\
\\
Solving this equation in $f$, we find
\begin{eqnarray*}	f=\frac{(1-u)(\alpha_{1}+\alpha_{2})-\sqrt{(u-1)^{2}(\alpha_{1}+\alpha_{2})^{2}-4(1-\alpha_{1}\alpha_{2}u)(\alpha_{1}\alpha_{2}-u)}}{2(1-\alpha_{1}\alpha_{2}u)}
\end{eqnarray*}
We define now the function $F:\mathbb{D}\longrightarrow\mathbb{C}$ by
\begin{eqnarray*}
F(z)\!\!\!&=&\!\!\!(u(z)-1)^{2}(\alpha_{1}+\alpha_{2})^{2}-4(1-\alpha_{1}\alpha_{2}u(z))(\alpha_{1}\alpha_{2}-u(z))\\
&=&\!\!\!((\alpha_{1}-\alpha_{2})^{2})u(z)^{2}+(4+4\alpha_{1}^{2}\alpha_{2}^{2}-2(\alpha_{1}+\alpha_{2})^{2})u(z)+(\alpha_{1}-\alpha_{2})^{2}
\end{eqnarray*}
Then $f$ represents an analytic function if and only if $F(\mathbb{D})$ is in the domain of some branch of the logarithm. In general, this does not occur. For instance, if $\alpha_{1}=1/2$ and $\alpha_{2}=1/4$, $F(\mathbb{D})$ contains a neighborhood of origin.
\\ 
\\
But, if for a given choice of the constants $\alpha_{1}$ and $\alpha_{2}$ we have $F(\mathbb{D})\cap [0,\infty)=\emptyset$, then there exists $\sqrt{F}$ and $f$ is analytic. To achieve this, it is sufficient that the solutions of $F(w)=t$ (with $t\geq0$) be complex numbers with $\left|w\right|\geq 1$.
\\
\\
First, we observe that $F(w)=t$ iff 
\begin{align} 
w=\frac{b-\sqrt{(-b)^{2}-4(\alpha_{1}-\alpha_{2})^{2}((\alpha_{1}-\alpha_{2})^{2}-t)}}{2(\alpha_{1}-\alpha_{2})^{2}}
\end{align}
where $b:=4\alpha_{1}^{2}\alpha_{2}^{2}+2(\alpha_{1}+\alpha_{2})^{2}-4$. Suppose now that $\beta$ is a real number in $(0,1)$ such that $4\beta^{4}+8\beta^{2}-4<0$. Making $\alpha_{1}\longrightarrow\beta^{+}$ and $\alpha_{2}\longrightarrow\beta^{-}$, the expression (2.2) tends to infinity and then, $\left|w\right|>1$ for all values $\alpha_{1}$ and $\alpha_{2}$ sufficiently close to $\beta$.\\
\\
We can summarize our analysis in the following theorem
\begin{theorem}
Let $\beta\in(0,1)$ be a real number such that $4\beta^{4}+8\beta^{2}-4<0$. Consider distinct real numbers $\alpha_{1},\alpha_{2}\in(0,1)$ sufficiently close to $\beta$. If $f:\mathbb{D}\longrightarrow \mathbb{D}$ is an analytic function with $f(0)=0$ and $f(\mathbb{D})\cap\left\{\alpha_{1},\alpha_{2}\right\}=\emptyset$, then $f$ satisfies the inequality (2.1) (with $k=2$). This inequality is sharp in the sense that there exists a function $f$ such that equality holds.    
\end{theorem}

\end{document}